\documentclass[reqno,a4paper,12pt]{amsart}

\usepackage{amssymb}

\numberwithin{equation}{section}
\date{\today}

\newcommand{\pd}{\partial}

\newcommand{\q}{\qquad}

\newcommand{\w}{\mathbf{w}}
\newcommand{\x}{\mathbf{x}}
\newcommand{\y}{\mathbf{y}}
\newcommand{\z}{\mathbf{z}}
\newcommand\n{\mathbf{n}}
\newcommand\s{\mathbf{s}}
\newcommand\vb{\mathbf{v}}

\newcommand{\R}{\mathbb{R}}

\newcommand{\C}{\mathbb{C}}

\newcommand{\Con}{C_0^\infty}

\newcommand{\grad}{\nabla}

\newtheorem{Lemma}{Lemma}
\newtheorem{Theorem}{Theorem}
\newtheorem{Corollary}{Corollary}

\newtheorem{Remark}{Remark}
\newtheorem{Example}{Example}
\begin{document}

\title[Hardy's inequality and curvature]{Hardy's inequality and curvature}
\author[A.~Balinsky]{A.~Balinsky}
\author[W.~D.~Evans]{W.D. Evans}
\address{Cardiff School of Mathematics\\
Cardiff University\\
Cardiff, CF24 4AG, Wales, UK}
\email{BalinskyA@cardiff.ac.uk}
\email{EvansWD@cardiff.ac.uk}
\author[R.~T.~Lewis]{R.T. Lewis}
\address{Department of Mathematics\\
         University of Alabama at Birmingham\\
         Birmingham, AL 35294-1170\\
         USA}
\email{rtlewis@uab.edu}



\begin{abstract}
{A Hardy inequality of the form
 \[
 \int_{\Omega} |\nabla f({\bf{x}})|^p d {\bf{x}} \ge \left(\frac{p-1}{p}\right)^p \int_{\Omega}
 \{1 + a(\delta, \partial \Omega)(\x)\}\frac{|f({\bf{x}})|^p}{\delta({\bf{x}})^p} d{\bf{x}},
 \]
 for all $f \in C_0^{\infty}({\Omega\setminus{\mathcal{R}(\Omega)}}),$ is considered for $p\in (1,\infty)$,
 where ${\Omega}$ is a domain in $\mathbb{R}^n$, $n \ge 2$, $\mathcal{R}(\Omega)$ is the \textit{ridge} of $\Omega$,
 and $\delta({\bf{x}})$ is the distance from ${\bf{x}} \in {\Omega} $
 to the boundary $ \partial {\Omega}.$ The main emphasis is on
 determining the dependance of $a(\delta, \partial {\Omega})$ on the geometric properties
 of $\partial {\Omega}.$ A Hardy inequality is also
 established for any doubly connected domain $\Omega$ in
 $\mathbb{R}^2$ in terms of a uniformization of $\Omega,$ that is,
 any conformal univalent map of $\Omega$ onto an annulus. }
\end{abstract}

\maketitle

\section {Introduction}

This paper is a contribution to the much studied Hardy inequality
\begin{equation}\label{Eq1}
    \int_{\Omega} |\nabla f(\x)|^p d \x \ge c(n,p,\Omega)
    \int_{\Omega} \frac{|f(\x)|^p}{\delta(\x)^p} d\x,\ \ \ f \in
    C_0^{\infty}(\Omega),
\end{equation}
where $\Omega$ is a domain (an open connected set) in $\R^n$,
 $n\ge 2$,  $1<p<\infty,$
and $\delta $ is the distance function
\[
\delta(\x):= \rm{dist}(\x, \R^n\setminus \Omega) = \inf_{\y \in
\R^n \setminus \Omega} |\x - \y|,\ \ \ \x \in \Omega.
\]
In order to put the problems we address in context and to
summarise our main results, we recall some of the highlights
amongst the known results to be found in the literature. For a
convex domain $\Omega$ in $\R^n, n\ge 2,$ the optimal constant in
(\ref{Eq1}) is
\begin{equation}\label{Eq2}
c(n,p,\Omega) = \left(\frac{p-1}{p}\right)^p;
\end{equation}
see \cite{MMP} and \cite{MatSob}. In all cases equality is only
achieved by $f=0.$ In the case $p=2$ the inequality was improved
by Br\'{e}zis and Marcus in \cite{BM} to one of the form
\begin{equation}\label{Eq3}
\int_{\Omega} |\nabla f(\x)|^2 d \x \ge \frac{1}{4}
    \int_{\Omega} \frac{|f(\x)|^2}{\delta(\x)^2} d\x +
    \lambda(\Omega) \int_{\Omega} |f(\x)|^2 d \x,
\end{equation}
where
\[
\lambda(\Omega) \ge \frac{1}{4 \rm{diam}(\Omega)^2}.
\]
Further improvements along these lines were made in \cite{HHL},
\cite{EL2}, \cite{T}, including ones for $p \in (1,\infty)$ in
\cite{EL2} and \cite{T}, and for Hardy-Sobolev inequalities in
\cite{FMT2}. Further pertinent references may be found in these
cited papers.

For non-convex domains, a sharp constant in (\ref{Eq1}) is not
known in general, but some sharp results were obtained in
\cite{FMT3}, \cite{D2} and \cite{T}. For a planar simply connected
domain $\Omega$, Ancona in \cite{A} proved the celebrated result
that
\begin{equation}\label{Eq4}
    c(2,2,\Omega) \ge \frac{1}{16}.
\end{equation}
By assuming certain ``quantifiable" degrees of convexity on a
simply connected, planar domain $\Omega $, Laptev and Sobolev in
\cite{LS} strengthened the Kobe \textit{one-quarter theorem} used
by Ancona in his proof and improved the lower bound in
(\ref{Eq4}). Other results of particular relevance to the present
paper are those in \cite{AL} for annular regions.

Our objective was to consider inequalities of the form
\begin{equation}\label{Eq5}
\int_{\Omega} |\nabla f(\x)|^p d \x  \ge c(n,p,\Omega)
\int_{\Omega} \left \{ 1 + a(\delta, \partial \Omega)(\x) \right
\} \frac{|f(\x)|^p}{\delta(\x)^p} d \x
\end{equation}
in which the function $a(\delta, \partial \Omega)$ depends on
$\delta$ and geometric properties of the boundary $\partial
\Omega$ of $\Omega$. We are particularly interested in domains
which are either convex or have convex complements. In these
cases, we determine $a(\delta,
\partial \Omega)$ explicitly in terms of $\delta$
and the principal curvatures of the boundary $\partial \Omega$ of
$\Omega$. Our analysis makes it necessary to consider the skeleton
$\mathcal{S}(\Omega)$ and ridge $\mathcal{R}(\Omega)$ of $\Omega$:
these will be defined in \S 2. A sample result is the following
special case of our Corollary~\ref{Cor2} where $\Omega$ is a
convex domain with $C^2$ boundary, $p=n=2$, and a condition on the
regularity of the ridge of $\Omega$ holds (see \ref{Sign}):
\begin{equation}\label{Eq6}
\int_{\Omega} |\nabla f(\x)|^2 d \x  \ge \frac{1}{4} \int_{\Omega}
\left \{ 1 + \left|\frac{2\kappa \delta}{1+\kappa
\delta}(\x)\right|\right \} \frac{|f(\x)|^2}{\delta(\x)^2} d \x,
\end{equation}
for $ f \in C_0^{\infty}(\Omega),$ where $\kappa $ is the
curvature of $\partial \Omega $. For $\Omega = B_R$, the open disc
of radius $R$ and center the origin, the condition (\ref{Sign})
holds, and the inversion $ \y = \x/|\x|^2,$ with $\rho = 1/R,$
yields
\begin{equation}\label{Eq7}
    \int_{\R^2 \setminus \overline{B_{\rho}}}|\nabla f(\y)|^2 d \y  \ge \frac{1}{4}
\int_{\R^2 \setminus \overline{B_{\rho}}} \left \{
-\frac{1}{|\y|^2} + \frac{1}{(|\y| -\rho)^2} \right \} |f(\y)|^2 d
\y.
\end{equation}
This inequality is given in \cite{AL}, Remark 1. A significant
feature of (1.6) with respect to (1.7) is that the presence of the
``alien" term $-1/|\y|^2$ in (1.7) is explained by the curvature
of the boundary. Other results in \cite{AL} are recovered from
theorems in Section 3 below by taking the convex sets involved
therein to be a ball. In Section 4 non-convex domains are
considered. Examples are given of Hardy inequalities on a torus
and on a 1-sheeted hyperboloid, which is unbounded with an
unbounded interior radius.

In Theorem~\ref{Thm3} we establish a Hardy inequality for any
doubly connected domain $\Omega$ in $\R^2$ in terms of a
uniformization of $\Omega$, i.e. any conformal univalent map of
$\Omega$ onto an annulus $B_R \setminus \overline{B_{\rho}}$ in
$\R^2.$ This is a rich source of examples of Hardy's inequalities
on non-convex domains. For example, Hardy's inequality is readily
derived for the domain
\[
\{z: \rho^2 < |\Phi(z)| < R^2\},\ \ \ z=x+iy,
\]
where $\Phi(z) = (z-1)(z+1)$. In this case $\sqrt{\Phi(z)}$ is an
appropriate uniformization.

The authors are grateful to Rupert Frank and Junfang Li for
comments on an earlier version of the paper which led to
significant improvements.

\section{Curvature and distance to the boundary}

The inequalities to be considered in the next section require the
determination of the Laplacian of the distance function in terms
of the principal curvatures of the boundary of the domain. We
first recall the following facts which may be found in \cite{EE},
Section 5.1. The \textit{skeleton} of a domain $\Omega$ is the set
\[
\mathcal{S}(\Omega):= \{ \x \in \Omega: {\rm{card}}\ N(\x) >1 \}
\]
where $N(\x) = \{ \y \in \pd \Omega: |\y-\x| = \delta(\x)\},$ the
set of  \textit{near points} of $\x$ on $\pd \Omega.$ The function
$\delta$ is differentiable at $\x$ if and only if $ \x \notin
\mathcal{S}(\Omega)$.  In $\Omega \setminus \mathcal{S}(\Omega),\
\nabla \delta $ is continuous and $ \nabla \delta (\x) = (\y -
\x)/|\y - \x|,$ where $ N(\x) = \{ \y\}.$ (When $N(\x)=\{\y\}$, we
sometimes abuse the notation and write $\y=N(\x)$.) Since $\delta$
is Lipschitz continuous, and hence differentiable almost
everywhere by Rademacher's Theorem, $\mathcal{S}(\Omega)$ is of
Lebesgue measure zero. In \cite{EE}, Corollary~5.1.4, it is shown
that if $\x\in\Omega$ and $\y\in N(\x)$ then
$N(\y+t[\x-\y])=\{\y\}$ for all $t\in (0,\lambda)$, where
$$
\lambda:=\sup\{t\in (0,\infty):\y\in N(\y+t[\x-\y])\}.
$$
The point $p(\x):=\y+\lambda[\x-\y]$ is called the {\it ridge
point} of $\x\in \Omega$ and
$\mathcal{R}(\Omega):=\{p(\x):\x\in\Omega\}$ is called the {\it
ridge} of $\Omega$. For further details and properties of
$\mathcal{S}(\Omega)$ and $\mathcal{R}(\Omega)$ we refer to
\cite{EE}, \S 5.1. In particular, note that $\mathcal{R}(\Omega)$
can be much larger than $\mathcal{S}(\Omega)$ and
$\mathcal{S}(\Omega)\subseteq \mathcal{R}(\Omega)\subseteq
\overline{\mathcal{S}(\Omega)}$. We shall be assuming throughout,
without further mention, that $\mathcal{R}(\Omega)$ is closed
relative to $\Omega$, and so
$\mathcal{R}(\Omega)=\overline{\mathcal{S}(\Omega)}$; it is proved
in \cite{EE}, Theorem 5.1.10, that this is equivalent to the
functions $p$ and $\delta\ o\ p $ being continuous on $\Omega$.
Note that in \cite{LN}, Theorem~1.1, it is proved that if $\Omega$
has a $C^{2,1}$ boundary (cf. next paragraph) then $\delta\ o\ p$
is Lipschitz continuous as a function defined on the boundary.

Let $\Omega$ be a domain in $\R^n, n \ge 2,$ with boundary $\pd
\Omega \in C^2$: this means that locally, after a rotation of
co-ordinates, $\pd \Omega$ is the graph
$x_n=\phi(x_1,x_2,\cdots,x_{n-1})$ of a function $\phi \in C^2.$
We consider a change of co-ordinates
\[
\Gamma: (s^1,s^2,\cdots,s^n) \rightarrow \x = (x^1,x^2,\cdots,x^n)
\]
defined for $\x \in \Omega$ by
\begin{equation}\label{Eq2.1}
\x = {\bf{\gamma}}(s^1,s^2,\cdots,s^{n-1}) + s^n
\n(s^1,s^2,\cdots,s^{n-1}).
\end{equation}
Here ${\bf{\gamma}}(s^1,s^2,\cdots,s^{n-1})\in \partial \Omega $,
and $\n( s^1,s^2,\cdots,s^{n-1})$ is the internal unit normal to
$\pd \Omega$ at $ {\bf{\gamma}}(s^1,s^2,\cdots,s^{n-1})$, i.e.
pointing in the direction of $\x$. The co-ordinates
$(s^1,s^2,\cdots,s^{n-1})$ are chosen with respect to principal
directions through the (unique) near point $N(\x)$ of $\x$ on $\pd
\Omega$, such that, with $ \s'=(s^1,s^2,\cdots,s^{n-1})$ and
\[
\frac{\pd \gamma}{\pd s^i} =: \vb_i = (v^1_i,v^2_i,\cdots,v^n_i),\
\ \ i=1,2,\cdots,n-1,
\]
we have
\begin{eqnarray}\label{Eq2.2}
  \langle \vb_i,\vb_j \rangle &=& \delta_{ij},\ \ \ \langle \vb_i, \n \rangle = 0,\ \ \  \nonumber \\
\frac{\pd \n(\s')}{\pd s^i} &=& \kappa_i(\s') \frac{\pd
    \gamma(\s')}{\pd s^i } = \kappa_i(\s') \vb_i(\s'),
\end{eqnarray}
where $\kappa_i, i=1,2,\cdots,n-1,$ are the \textit{principal
curvatures} of $\pd \Omega$ at the near point $\y$ of $\x$, and
the angular notation denotes scalar product. In (\ref{Eq2.2}), the
signs of the principal curvatures are determined by the direction
of the normal $\n$. If $\Omega$ is convex, the principal
curvatures of $\pd \Omega$ are non-positive, while if the domain
under consideration is $\bar\Omega^c = \R^n \setminus \bar\Omega$,
the principal curvatures are non-negative.

We set $s^n = \delta $; in (\ref{Eq2.1}), $\delta$ is equal to the
distance $\delta(\x)$ of $\x$ to $\pd \Omega$.

The following result (with $g(t)=t$) may be found in Gilbarg and
Trudinger~\cite{GT}, Lemma~14.17, for points close to the $C^2$
boundary of a bounded domain $\Omega$. For our reader's
convenience, we give our proof, which is designed for our needs.
Note the proof of Lemma~2.2 in \cite{LLL}, from which it follows
that if $\Omega$ has a $C^2$-boundary, then $\delta\in C^2$ on
$\Omega\setminus \mathcal{R}(\Omega)$. We have seen that
$\mathcal{S}(\Omega)$ is of zero measure and hence so is
$\mathcal{R}(\Omega)$ if $\mathcal{S}(\Omega)$ is closed. We
caution the reader that in \cite{GT} and \cite{LLL} computations
are made with respect to the outward unit normal (rather than the
inward unit normal as in this paper) causing a different sign for
the principal curvatures $\kappa_i$, $i=1,\dots,n-1$.

\begin{Lemma}\label{Lem1}
Let $\Omega$ be a domain in $\R^n$, $n\ge 2$, with $C^2$ boundary,
and set $\delta(\x) : =\rm{dist}(\x,\pd \Omega).$ Then $\delta\in
C^2(\Omega\setminus \mathcal{R}(\Omega))$, and for $g(\x) =
g(\delta(\x)), g \in C^2(\R^+),$
\begin{equation}\label{Eq2.3}
    \Delta_{\x} g(\x) = \frac{\pd^2g}{\pd \delta^2}(\x) +
    \sum_{i=1}^{n-1}\left(\frac{\kappa_i}{1+\delta \kappa_i}\right)
    \frac{\pd g}{\pd \delta}(\x),
\end{equation}
where the $\kappa_i$ are the principal curvatures of $ \pd \Omega$
at the near point $N(\x)$ of $\x$. The equation (\ref{Eq2.3})
holds for all $\x\in\Omega \setminus \mathcal{R}(\Omega)$.
\end{Lemma}

\begin{proof}

From (\ref{Eq2.1}), for $i=1,2,\dots,n ,$
\begin{equation}\label{Eq2.4}
    \frac{\pd x^i}{\pd s^j} = \frac{\pd \gamma^i}{\pd s^j} + \delta \frac{\pd n^i}{\pd s^j},\ \ \
    j=1,2,\cdots n-1,\ \ \ \frac{\pd x^i}{\pd \delta} = n^i,
\end{equation}
and so, by (\ref{Eq2.2}),
\begin{equation}\label{Eq2.5}
    \frac{\pd \x}{\pd s^j} = (1+ \delta \kappa_j) \vb_j,\ j=1,2,\cdots, n-1,\ \ \ \frac{\pd \x}{\pd s^n} = \n.
\end{equation}
Therefore, on recalling that $s^n=\delta$
\begin{eqnarray}\label{Eq2.5A}
    \left(\begin{array}{ccc}
    1& \cdots& 0 \\
& \ddots &  \\
    0&\cdots& 1 \end{array} \right) &=&
    \left(\begin{array}{ccc}
    \frac{\pd x^1}{\pd s^1} &\cdots &   \frac{\pd x^1}{\pd s^n}\\
\vdots& &\vdots\\
     \frac{\pd x^n}{\pd s^1}&\cdots & \frac{\pd x^n}{\pd s^n} \end{array} \right)
     \left(\begin{array}{ccc}
    \frac{\pd s^1}{\pd x^1}& \cdots &   \frac{\pd s^1}{\pd x^n}\\
\vdots & &\vdots \\
     \frac{\pd s^n}{\pd x^1}&\cdots &  \frac{\pd s^n}{\pd x^n} \end{array}
     \right)\nonumber \\
&=&\left(\begin{array}{cccc}
   (1+\delta \kappa_1) v_1^1& \cdots &  (1+\delta \kappa_{n-1}) v_{n-1}^1& \ n^1\\
\vdots & & \vdots&\vdots \\
(1 +\delta \kappa_1) v_1^n&\cdots &  (1+\delta \kappa_{n-1}) v_{n-1}^n & n^n\end{array} \right)
%   [1+\delta \kappa_1]v_1^n &\cdots & [1+\delta \kappa_{n-1}]v_{n-1}^n \ n^n \end{array} \right)
 \nonumber \\
& & \times \left(\begin{array}{ccc}
    \frac{\pd s^1}{\pd x^1} &\cdots   &\frac{\pd s^1}{\pd x^n}\\
\vdots & & \vdots \\
     \frac{\pd s^n}{\pd x^1}&\cdots & \frac{\pd s^n}{\pd x^n} \end{array}
     \right)
\end{eqnarray}
It follows from (\ref{Eq2.2}) that
\begin{equation}\label{Eq2.6A}
\left(\begin{array}{ccc}
    (1+\delta \kappa_1)^{-1}v_1^1& \cdots &  (1+\delta \kappa_1)^{-1}v_1^n \\
    \vdots & & \vdots \\
     (1+\delta \kappa_{n-1})^{-1}v_{n-1}^1& \cdots &  (1+\delta \kappa_{n-1})^{-1}v_{n-1}^n \\
n^1& \cdots & n^n
     \end{array} \right)= \left(\begin{array}{cc}
    \frac{\pd s^1}{\pd x^1} \cdots   \frac{\pd s^1}{\pd x^n}\\
\vdots \\
     \frac{\pd s^n}{\pd x^1}\cdots  \frac{\pd s^n}{\pd x^n} \end{array}
     \right)
\end{equation}
Therefore, for $j=1,2,\cdots n, \ \ i=1,2,\cdots n-1,$
\begin{equation}\label{Eq2.7}
    \frac{\pd s^i}{\pd x^j} = [1+\delta \kappa_i]^{-1} v_i^j,\ \ \
    \frac{\pd \delta}{\pd x^j} = n^j
\end{equation}
and, employing the usual summation convention,
\begin{eqnarray*}
  \frac{\pd^2 \delta}{\pd (x^j)^2} &=& \frac{\pd n^j}{\pd s^i} \frac{\pd s^i}{\pd x^j}\\
   &=& \sum_{i=1}^{n-1} [1+ \delta \kappa_i]^{-1}v_i^j\frac{\pd
   n^j}{\pd s^i} + \frac{\pd n^j}{\pd \delta} n^j.
\end{eqnarray*}
Consequently
\begin{eqnarray}\label{Eq2.7A}
   \Delta \delta &=& \sum_{j=1}^{n}\left \{\sum_{i=1}^{n-1}[1+\delta \kappa_i]^{-1} v_i^j \frac{\pd n^j}{\pd s^i} +
   n^j \frac{\pd n^j}{\pd \delta}\right \} \nonumber\\
   &=& \sum_{i=1}^{n-1}[1+\delta\kappa_i]^{-1} \langle \vb_i,
   \frac{\pd \n}{\pd s^i}\rangle + \langle \n, \frac{\pd \n}{\pd
   \delta}
   \rangle \nonumber \\
   &=& \sum_{i=1}^{n-1}\kappa_i[1+\delta\kappa_i]^{-1}
\end{eqnarray}
by (\ref{Eq2.2}). From the Chain Rule, we have
\[
\frac{\pd g}{\pd x^j} = \frac{\pd g}{\pd s^i}\frac{\pd s^i}{\pd
x^j} + \frac{\pd g}{\pd \delta} \frac{\pd \delta}{\pd x^j} =
\frac{\pd g}{\pd \delta} \frac{\pd \delta}{\pd x^j}
\]
and, on using (\ref{Eq2.7}),
\begin{eqnarray*}
  \Delta_{\x} g &=& \frac{\pd}{\pd s^k}\left[ \frac{\pd g}{\pd \delta} \frac{\pd \delta}{\pd x^j}\right] \frac{\pd s^k}{\pd x^j}\\
  &=& \left[ \frac{\pd^2 g}{\pd s^k \pd \delta} n^j + \frac{\pd
  g}{\pd \delta} \frac{\pd n^j}{\pd s^k} \right] \frac{\pd
  s^k}{\pd x^j} \\
   &=& \frac{\pd^2 g}{\pd (\delta^2)} + \frac{\pd g}{\pd
   \delta} \sum_{j=1}^{n}\sum_{k=1}^{n-1}\kappa_k v_k^j [1+\delta \kappa_k]^{-1} v_k^j \\
   &=& \frac{\pd^2 g}{\pd (\delta^2)} + \frac{\pd g}{\pd
   \delta}  \sum_{k=1}^{n-1}\kappa_k [1+\delta \kappa_k]^{-1}.
\end{eqnarray*}
The lemma is therefore proved.
\end{proof}

\bigskip

\begin{Remark}\label{Remsigns2}
It follows from (\ref{Eq2.7}) that the terms $\left[\kappa_i/(1+
\delta \kappa_i)\right] (\y)$, $\y=N(\x)$, $ i=1,2,\cdots,n-1,$ in
(\ref{Eqsigns}), are the principal curvatures of the level surface
of $\delta$ through $\x$ at $\x$. Furthermore,
$\frac{1}{n-1}\sum_{i=1}^{n-1} \left[\kappa_i/(1+ \delta
\kappa_i)\right] (\y)$ is the mean curvature of this level surface
at $\x$.
\end{Remark}

\bigskip

\begin{Remark}\label{Remskeleton} If $\Omega$ is convex, then $S(\bar\Omega^c)=R(\bar\Omega^c)=
\emptyset$.
\end{Remark}

\bigskip

\begin{Remark}\label{RemSigns}
If $\Omega$ is a convex domain with a $C^2$-boundary, we have
noted that the principal curvatures of $\pd \bar\Omega^c$ are
non-negative and hence
\begin{equation}\label{Eqsigns}
\Delta \delta (\x)= \tilde{\kappa}(\y):=  \sum_{i=1}^{n-1}
\left[\kappa_i/(1+ \delta \kappa_i)\right] (\y) \ge 0 \ \ \rm{for
\ all} \ \ \x \in \bar\Omega^c
\end{equation}
where $\{\y\}=N(\x)$.

We claim that for $\Omega$ convex, we also have
\begin{equation}\label{Eqsigns2}
\Delta \delta (\x)= \tilde{\kappa}(\y):=  \sum_{i=1}^{n-1}
\left[\kappa_i/(1+ \delta \kappa_i)\right] (\y) \le 0 \ \  \rm{for
\ all}\ \   \x \in \Omega \setminus \mathcal{R}(\Omega).
\end{equation}
To see this, let $\x_0$ be an arbitrary point in $\Omega, \y_0 =
N(\x_0)$, and let $ h(s)$ be a principal curve through $\y_0$ on
$\pd \Omega$ with curvature $\kappa$ at $\y_0$: thus
\[
h(0) =\y_0,\ \ |h'(0)|=1,\ \ h''(0) = -\kappa \n(0),
\]
where $\n$ is the {\bf inward normal} to $\pd \Omega$ at $\y_0.$ The
function
\[
f(s) := | h(s) - \x_0|^2
\]
has a minimum at $s=0$ and so at $s=0,$ we have
\begin{eqnarray*}
  f'(s) &=& 2 h'(s) \cdot [h(s)-\x_0] = 0 \\
  f''(s) &=& 2 h''(s) \cdot [h(s)- \x_0] + 2 |h'(s)|^2\ge 0.
\end{eqnarray*}
Consequently, since $  [h(0) -\x_0 ]= -\delta \n ,$ we have
\[
1+ \delta \kappa \ge 0.
\]
This is true for all principal directions and so, since the
principal curvatures are non-positive, our claim (\ref{Eqsigns2})
is established.

\end{Remark}

\bigskip

\begin{Remark}\label{Rem4}
If $\kappa_i(\y)\ge 0$, then
\begin{equation}\label{positivity}
1+\delta(\x)\kappa_i(\y) \ge 1,\q N(\x)=\{\y\}.
\end{equation}
Therefore, if $\Omega$ is the complement of a closed convex domain
with $C^2$ boundary, then (\ref{positivity}) holds for all $i$
throughout $\Omega$, since then $\mathcal{R}(\Omega) =
\mathcal{S}(\Omega) = \emptyset$.

Suppose $\Omega$ is convex and $\partial\Omega\in C^2$. Then for
all $\x\in \Omega\setminus \mathcal{R}(\Omega)$ and for all $i$
\begin{equation}\label{positivity2}
1+\delta(\x)\kappa_i(\y)=1-\delta(\x)|\kappa_i(\y)|>0, \q
N(\x)=\{\y\}.
\end{equation}
 This is proved as follows.
We saw in Remark~\ref{RemSigns} that $1+\delta \kappa_i\ge 0$.
Suppose for some $i$ and some $\x\in\Omega$ with $N(\x)=\{\y\}$,
 that $\delta(\x)=1/|\kappa_i(\y)|$. If $\x$ lies outside $\mathcal{R}(\Omega)$,  then it follows from \cite{EE}, Corollary~5.1.4,
 that there exists a point $\w\in \Omega\setminus \mathcal{R}(\Omega)$ on the ray from $\y$ through $\x$ such that
 $N(\w)=\{\y\}$ and $\delta(\w)>\delta(\x)$. But, this would imply that $1+\delta(\w)\kappa_i(\y)<0$
 which is a contradiction.

\end{Remark}

\bigskip

\begin{Remark}\label{Rem5}
Suppose $\Omega$ is convex with $\partial\Omega\in C^2$. Then
$\x\in \mathcal{R}(\Omega)\setminus \mathcal{S}(\Omega)$ if and
only if for $N(\x)=\{\y\}$
\begin{equation}\label{equality}
1+\delta(\x)\kappa_i(\y)=0,\q \text{for some\ \ \ } i.
\end{equation}
The ridge in this case has zero Lebesgue measure.
\end{Remark}

\bigskip

\section{Inequalities inside and outside domains}

We first establish the following general inequality

\begin{Theorem}\label{ThmP} Let $\Omega \subset\R^n, n \ge 2,$ be a domain having a ridge $\mathcal{R}(\Omega)$ and a sufficiently smooth boundary
 for Green's formula to hold. Let $\delta(\x) = \rm{dist}(\x, \R^n \setminus \Omega).$ Then for
 all $f \in C_0^{\infty}(\Omega\setminus \mathcal{R}(\Omega))$ and $p \in (1,\infty),$
\begin{equation}\label{EE}
\int_\Omega |\grad \delta\cdot \grad f|^pd\x \ge
\left(\frac{p-1}{p}\right)^p \int_\Omega \left \{1 - \frac{p \delta
\Delta\delta}{p-1} \right \} \frac{|f|^p}{\delta^{p}}d\x.
\end{equation}

\end{Theorem}
\begin{proof}
For any vector field $V$ we have the identity
\begin{equation}\label{BE1}
 \int_{\Omega}
({\rm div} V) |f|^p d\x = -p \left[\rm{Re}
   \int_{\Omega}(V\cdot \nabla f) |f|^{p-2}
   \overline{f} d \x \right ]
   \end{equation}
   for all $f\in\Con(\Omega\setminus \mathcal{R}(\Omega))$.
Choose
\[ V = -p \nabla \delta / \delta^{p-1}.
\]
Then, for any $ \varepsilon >0,$
\begin{eqnarray*}
  \int_\Omega \rm{div} V|f|^p d\x
  & \le & p^2 \left(\int_\Omega |\nabla \delta \cdot \nabla f|^p d \x \right)^{1/p}
  \left(\int_\Omega \frac{|f|^p}{\delta^p} d \x\right)^{1-1/p} \\
   &\le & p \varepsilon^p \int_\Omega |\nabla \delta \cdot \nabla f|^p d
   \x + p(p-1)\varepsilon^{-p/(p-1)}\int_\Omega  \frac{|f|^p}{\delta^p}d
   \x
\end{eqnarray*}
which gives, since $ \rm{div} V = (p-1) p \delta^{-p} -p
\delta^{1-p} \Delta \delta $ for $\x\in \Omega\setminus
\mathcal{R}(\Omega)$,
\[
\int_\Omega |\nabla \delta \cdot \nabla f|^p d
   \x \ge \varepsilon^{-p} \int_\Omega \left[(p-1) -
   (p-1)\varepsilon^{-p/(p-1)}- \delta \Delta \delta \right]\frac{|f|^p}{\delta^p} d
   \x.
\]
The proof of (\ref{EE}) is completed on choosing $ \varepsilon =
[p/(p-1)]^{\frac{(p-1)}{p}}.$

\end{proof}

\begin{Corollary} If $\partial \Omega\in C^2$, then for all $f\in\Con(\Omega\setminus \mathcal{R}(\Omega))$
\begin{equation}\label{EE2}
\int_\Omega |\grad \delta\cdot \grad f|^pd\x \ge
\left(\frac{p-1}{p}\right)^p \int_\Omega \left \{1 - \frac{p \delta
\tilde \kappa}{p-1} \right \} \frac{|f|^p}{\delta^{p}}d\x
\end{equation}
where $\tilde\kappa:=\sum_{i=1}^{n-1}\frac{\kappa_i}{1+\delta\kappa_i}$.
\end{Corollary}
\begin{proof}
The proof follows from Lemma~\ref{Lem1} and Theorem~\ref{ThmP}.

\end{proof}

In many cases, we are able to prove an inequality for all $f\in\Con(\Omega)$.

\begin{Theorem}\label{ThmP2} Assume the hypothesis of Theorem~\ref{ThmP}. Let $\mathcal{R}(\Omega)$ be the intersection of a
decreasing family of open neighborhoods $\{S_\epsilon:
\epsilon>0\}$ with smooth boundaries, and let $\eta_\epsilon(\x)$
denote the unit inward normal at $\x\in\partial S_\epsilon$. If
\begin{equation}\label{Sign}
(\grad\delta\cdot \eta_\epsilon)(\x) \ge 0, \q \x\in \partial S_\epsilon
\end{equation}
for all $\epsilon$ sufficiently small and
\begin{equation}\label{Sign2}
\frac{p-1}{p}\ge [\delta \Delta\delta](\x),\q \x\in \Omega\setminus R(\Omega),
\end{equation}
then (\ref{EE}) holds for all $f\in\Con(\Omega)$.
\end{Theorem}
\begin{proof} We proceed as in (\ref{BE1}), but now with $f\in\Con(\Omega)$, and account for the contribution of the boundary of $S_\epsilon$. On using (\ref{Sign}) we have for all $f\in\Con(\Omega)$
\begin{equation}\label{BE1A}
 \int_{\Omega\setminus S_\epsilon}
(\rm{div} V) |f|^p d\x \le -p \left[\rm{Re}
   \int_{\Omega\setminus S_\epsilon}(V\cdot \nabla f) |f|^{p-2}
   \overline{f} d \x \right ].
   \end{equation}
   On proceeding as in the proof of Theorem~\ref{ThmP} we obtain
$$\begin{array}{rl}
 \int_{\Omega} |\nabla \delta \cdot \nabla f|^p d\x \ge&
 \int_{\Omega\setminus S_\epsilon} |\nabla \delta \cdot \nabla f|^p d\x\\
 \ge & \left(\frac{p-1}{p}\right)^{p} \int_{\Omega}
   \left[1- \frac{p}{p-1}\delta \Delta \delta \right]\frac{|f|^p}{\delta^p} \chi_{_{\Omega\setminus S_\epsilon}} d\x. \end{array}
 $$
 The proof concludes on using (\ref{Sign2}) and the monotone convergence theorem.

\end{proof}

\begin{Corollary} \label{Cor2} Suppose that the hypothesis of Theorem~\ref{ThmP2} is satisfied and that $\Omega$ is
convex with a $C^2$ boundary. Then for all $f\in\Con(\Omega)$
\begin{equation}\label{Convex}
\int_\Omega |\grad \delta\cdot \grad f|^pd\x \ge
\left(\frac{p-1}{p}\right)^p \int_\Omega \left \{1 + \frac{p \delta |\tilde\kappa|}{p-1} \right \}
\frac{|f|^p}{\delta^{p}}d\x.
\end{equation}

\end{Corollary}
\begin{proof}
Since $\Omega$ is convex, it follows from Lemma~\ref{Lem1} that
$-\Delta\delta = -\tilde\kappa\ge 0$ for $\x\in \Omega\setminus
\mathcal{R}(\Omega)$; see Remark~\ref{RemSigns}. Therefore,
(\ref{Sign2}) must hold and the result then follows from
Theorem~\ref{ThmP2}.

\end{proof}

\begin{Corollary}\label{Cor1a} Let $\Omega$ be a ball
$B_R:=\{\x\in\R^n:|\x|<R\}$. Then for $p>1,$
\begin{equation}\label{Eq12}
\int_{B_R}| \grad f|^pd\x -\left(\frac{p-1}{p}\right
)^p\int_{B_R}\frac{|f|^p}{\delta^p}d\x\ge
\left(\frac{p-1}{p}\right)^{p-1}\int_{B_R}\frac{(n-1)|f|^p}{|\x|\delta^{p-1}}d\x
\end{equation}
for all $f\in\Con(B_R)$.
\end{Corollary}
\begin{proof}
In this case we have that
$\mathcal{R}(B_R)=\overline{\mathcal{S}(B_R)}=\{0\}$ and $\delta
=R-|\x|$. We now have that $S_\epsilon=B_\epsilon$ and on
$\partial S_\epsilon$,
$\grad\delta=-\frac{\x}{|\x|}=\eta_\epsilon$. Therefore
(\ref{Sign}) holds implying that (\ref{Eq12}) is valid since
$|\tilde\kappa|=(n-1)/|\x|$.
\end{proof}

\bigskip

In \cite{FMT}, Theorem 3.1 , it is proved that for $\Omega \subset
\R^n$ convex,
\begin{equation}\label{FMT}
\int_\Omega |\grad u|^2d\x - \frac{1}{4}\int_\Omega
\frac{|u|^2}{\delta^2}d\x \ge c_{\alpha}
D_{int}^{-(\alpha+2)}\int_{\Omega} \delta^{\alpha}  |u|^2 d\x,
\end{equation}
for any $ \alpha > -2,$ where
\[
c_{\alpha} = \left \{ \begin{array} {ll} 2^{\alpha}(2 \alpha +3),
 & \ \rm{if}\ \ \alpha \ge -1 \\
2^{\alpha}(\alpha+2)^2, & \ \rm{if}\ \ \alpha \in
(-2,-1)\end{array} \right .
\]
and $D_{int}:=2\sup\{\delta(\x):\x\in\Omega\}$. A comparison of
the right-hand side of (\ref{FMT}), when $\Omega = B_R$, with that
in the case $p=2$ of (\ref{Eq12}),  is now made to seek further
evidence of the significance of the curvature in these
inequalities. Set $ \alpha = -2 + \varepsilon.$ Then the terms to
be compared from (\ref{Eq12}) and (\ref{FMT}), respectively, are $
I_1 = (n-1)/2\delta(\x) |\x| $ and $I_2 =
c_{\alpha}D_{int}^{-(\alpha+2)} \delta^{\alpha}(\x)$, with $
\delta(\x) = R-|\x|$. It is readily shown that
\[
I_1 - I_2 \ge \left \{ \begin{array}{ll} \frac{2n-1-2
\varepsilon}{4 \delta |\x|},\ & \rm{if} \ \varepsilon \ge 1,\\
\frac{(2n-2)R - (2n-1)|\x|}{4 \delta^2 |\x|}, \ & \rm{if}\
0<\varepsilon <1. \end{array} \right.
\]
A similar comparison can be made in the $L^p$ case using Theorem
3.2 of \cite{FMT} with $p=q $ and $ \alpha > -p.$ Also see
\cite{FMT2}.

\begin{Example}
{\it The infinite cylinder}. Let $\Omega=B_1(0)\times\R$, where
$B_1(0)$ is the unit ball, center the origin, in $\R^2$. Clearly,
$\Omega$ is convex and $R(\Omega)$ is the $z$-axis. The distance
function is $\delta = 1-\sqrt{x^2+y^2}$,
$$
\grad\delta = -(x,y,0)(1-\delta)^{-1},\ \ \ \  \eta=-(x,y,0)/(1-\delta)\ \ \text{on}\ \ \ \partial S_\epsilon,
$$
and $\Delta \delta=\frac{-1}{1-\delta}$, where $S_\epsilon:=\{\x=(x,y,z)\in\Omega:x^2+y^2<\epsilon\}$.
Therefore, (\ref{Convex}) holds for this cylinder with
$|\tilde\kappa|=|\Delta \delta|=\frac{1}{1-\delta}$.

\end{Example}

\begin{Theorem}\label{ThmP2A} Under the conditions of Theorem
\ref{ThmP}, we have for all $ f \in C_0^{\infty}(\Omega\setminus
\mathcal{R}(\Omega)),$
\begin{eqnarray}\label{Eq13}
    \int_{\Omega} |\nabla f(\x)|^2 d \x \ge
    \frac{1}{4}\int_{\Omega}\left \{\frac{(n-2)^2}{|\x|^2}\right.
    &+& \frac{(1+2|\delta     \Delta \delta|)}{\delta^2} \nonumber \\
  & +& \left.  2(n-2) \frac{\x \cdot \nabla
    \delta}{|\x|^2 \delta}     \right \} |f(\x)|^2 d \x.  \nonumber \\
\end{eqnarray}
In particular, if $ \Omega $ is a convex domain with a $C^2$
boundary, the conditions of the theorem are met and $ \Delta
\delta = \sum_{i=1}^{n-1} \kappa_i/(1+\delta \kappa_i) $ in
$\Omega\setminus \mathcal{R}(\Omega)$.
\end{Theorem}
\begin{proof}
Let
\[
V(\x) = -2 \frac{\nabla  \delta(\x)}{\delta(\x)} + 2(n-2)
\frac{\x}{|\x|^2}.
\]
Then
\[
\rm{div} V(\x) = \frac{2}{\delta(\x)^2} - \frac{2 \Delta
\delta(\x)}{\delta(\x)} + \frac{2(n-2)^2}{|\x|^2} \ge 0
\]
and
\[
\frac{1}{4}|V(\x)|^2 = \frac{1}{\delta^2} + \frac{(n-2)^2}{|\x|^2}
- 2(n-2) \frac{\x \cdot \nabla \delta}{|\x|^2 \delta}.
\]

For any $\varepsilon > 0,$ and $f\in\Con(\Omega\setminus
\mathcal{R}(\Omega))$
\begin{eqnarray*}
  \int_{\Omega} (\rm{div} V) |f|^2 d \x &=& -2 {\rm{Re}}\left[\int_{\Omega} (V \cdot \nabla f) \overline{f} d \x\right] \\
   & \le & 2 \left(\int_{\Omega} |\nabla f|^2 d\x \right)^{1/2}\left(\int_{\Omega}|V|^2 |f|^2 d\x
   \right)^{1/2}\\
   & \le & \varepsilon^2 \int_{\Omega} |\nabla f |^2 d \x +
   \varepsilon^{-2}\int_{\Omega}|V|^2 |f|^2 d\x.
\end{eqnarray*}
The result follows on choosing $\varepsilon = 2.$

\end{proof}

When $\Omega = B_R$, we have on substituting in Theorem
\ref{ThmP2A}, $ \delta(\x) = R-|\x|,\  \kappa_i(\x) = -1/R,
i=1,2,\cdots ,n-1 $, and so $ \Delta \delta(\x) = -(n-1)/|\x| $
from Lemma \ref{Lem1} (or by direct calculation),
\begin{equation}\label{Eq14}
    \int_{B_R} |\nabla f(\x)|^2 d\x \ge \frac{1}{4} \int_{B_R}
    \left \{ \frac{(n-2)^2}{|\x|^2} + \frac{1}{\delta(\x)^2} +
    \frac{2}{|\x| \delta(\x)}\right \} |f(\x)|^2 d \x
\end{equation}
which is given in Corollary 2 in \cite{AL}. Note that (\ref{Eq14}) is valid for all $f\in\Con(B_R)$ -- see Corollary~\ref{Cor1a}.

\bigskip

The application of Lemma \ref{Lem1} to Theorem \ref{ThmP} also
yields the following Hardy inequality in the complement of a
closed convex domain. Recall that in this case
$\mathcal{R}(\R^n\setminus\bar\Omega)=\emptyset$.

\begin{Theorem}\label{Thm5A} Let $\Omega \subset \R^n, n \ge 2,$
be convex with a $C^2$ boundary. Then for all $f \in
C_0^{\infty}(\R^n \setminus \bar\Omega),$
\begin{equation}\label{Des1}
    \int_{\R^n \setminus \bar\Omega} |\nabla f(\x)|^p d \x \ge
    \left(\frac{p-1}{p}\right)^p\int_{\R^n \setminus \bar\Omega}\left \{1 - \frac{p\tilde{\kappa}\delta}{p-1}\right \} \frac{|f(\x)|^p}{\delta(\x)^p}d
    \x,
\end{equation}
where $  \ \tilde{\kappa} = \sum_{i=1}^{n-1} \frac{\kappa_i}{1+
    \delta \kappa_i} \ge 0.$
\end{Theorem}

Note that if $\Omega = B_{\rho}$, the integrand on the right-hand
side of (\ref{Des1}) is non-negative if and only if
\[
|\x| \le 2\frac{(n-1) \rho}{2n-3}.
\]

The following is another form of Hardy inequality, reminiscent of
that derived in \cite{BE}, Theorem 3.1.

\begin{Theorem}\label{Thm5} Let $\Omega$ be a convex domain in $\R^n$
with a $C^2$ boundary. Then for all
$f\in\Con(\R^n\setminus\bar\Omega)$,
$$
\int_{\R^n\setminus\bar\Omega}\delta^p|\grad\delta\cdot\grad
f|^pd\x\ge \frac{1}{p^p}\int_{\R^n\setminus
\bar\Omega}[1+p\tilde\kappa \delta]|f|^pd\x,
$$
where $  \ \tilde{\kappa} = \sum_{i=1}^{n-1} \frac{\kappa_i}{1+
    \delta \kappa_i} \ge 0.$
\end{Theorem}
\begin{proof}
The proof follows the lines of that of Theorem 3.1 in \cite{BE}.
From (\ref{BE1}),
\begin{eqnarray*}
   \int_{\R^n\setminus\bar\Omega} (\rm{div} V) |f|^p d\x
   &\le & p \left(\int_{\R^n\setminus\bar\Omega}|V\cdot \nabla f|^p
   d\x\right)^{1/p} \left(\int_{\R^n\setminus\bar\Omega}|f|^p d\x
   \right)^{(p-1)/p} \\
   &\le & \varepsilon^p \int_{\R^n\setminus\bar\Omega}|V\cdot \nabla f|^p
   d\x + (p-1) \varepsilon^{-p/(p-1)}\int_{\R^n\setminus\bar\Omega}|f|^p
   d\x.
\end{eqnarray*}
On choosing $V= \delta^2$, we have
\begin{eqnarray}\label{Eq5}
2^p\epsilon^p\int_{\R^n\setminus\bar\Omega}\delta^p|\grad\delta\cdot\grad
f|^p d\x
&+& (p-1)\epsilon^{-\frac{p}{p-1}}\int_{\R^n\setminus\bar\Omega} |f|^pd\x \nonumber \\
&\ge & \int_{\R^n\setminus\bar\Omega} [\Delta\delta^2]|f|^pd\x \nonumber \\
& = & 2\int_{\R^n\setminus\bar\Omega}[1+ \tilde{\kappa} \delta ]
d\x
\end{eqnarray}
by Lemma 1. Hence, as in (3.6) of \cite{BE},
$$\begin{array}{rl}
2^p\int_{\R^n\setminus\bar\Omega} \delta^p|\grad\delta\cdot \grad
f|^pd\x
\ge& K(\epsilon)\int_{\R^n\setminus\bar\Omega} |f|^pd\x\\
&+2\epsilon^{-p}\int_{\R^n\setminus\bar\Omega} \tilde\kappa \delta
|f|^pd\x
\end{array}
$$
where
$$
K(\epsilon)=2\epsilon^{-p}-(p-1)\epsilon^{-\frac{p^2}{p-1}}
$$
has a maximum value of $(2/p)^p$ at $\epsilon=(p/2)^{(p-1)/p}$.
The proof is completed by making the substitution for this value
of $\epsilon$.

\end{proof}

When $p=2$, it is readily shown that the substitution $ u = \delta
f$ in Theorem \ref{Thm5} yields (\ref{Des1}).

\bigskip

\begin{Example} If $\Omega=B_{\rho}$ in Theorem \ref{Thm5}, then
$$
\int_{\R^n\setminus \overline{B_{\rho}}}(|\x|-\rho)^p|\grad f|^pd\x\ge
\frac{1}{p^p}\int_{\R^n\setminus \overline{B_{\rho}}}\left
[1+p(n-1)\frac{|\x|-\rho}{|\x|}\right]|f|^pd\x
$$
for all $f\in\Con(\R^n\setminus \overline{B_{\rho}})$.
\end{Example}

\bigskip

We have the following analogue of Theorem 1 in \cite{AL} for an
annulus bounded by convex domains.

\begin{Theorem}\label{Thm6} Let $\Omega_1$, $\Omega_2, $ be convex
domains in $\R^n, n\ge 2,$ with $C^2$ boundaries and $
\bar{\Omega}_1 \subset \Omega_2.$ For $\x \in \Omega := \Omega_2
\setminus \bar{\Omega}_1$ denote the distances of $\x$ to $\pd
\Omega_1, \pd \Omega_2$ by $\delta_1, \delta_2,$ respectively.
Then for all $f\in\Con(\Omega\setminus \mathcal{R}(\Omega))$
\begin{eqnarray}\label{Eq6}
  \int_{\Omega_2 \setminus \bar{\Omega}_1} |\nabla f(\x)|^2 d \x  &\ge & \frac{1}{4} \int_{\Omega_2 \setminus \bar{\Omega}_1}\left \{ \frac{(n-1)(n-3)}{|\x|^2}
  + \frac{1}{\delta_1^2} +  \frac{1}{\delta_2^2} \right.\nonumber \\
   & -& \frac{2 \Delta \delta_1}{\delta_1} - \frac{2 \Delta
   \delta_2}{\delta_2}- \frac{2\nabla \delta_1 \cdot \nabla
   \delta_2}{\delta_1 \delta_2} \nonumber \\
   & + & \left.2(n-1)\frac{\x \cdot \nabla \delta_1}{|\x|^2 \delta_1} +2(n-1)\frac{\x \cdot \nabla \delta_2}{|\x|^2
   \delta_2}\right \} |f(\x)|^2 d \x. \nonumber \\
\end{eqnarray}
\end{Theorem}
\begin{proof}
The starting point is again
\[
\int_{\Omega_2 \setminus \bar{\Omega}_1} (\rm{div}V) |f(\x)|^2 d\x
\le \varepsilon^2 \int_{\Omega_2 \setminus \bar{\Omega}_1} |\nabla
f |^2 d \x +
   \varepsilon^{-2}\int_{\Omega_2 \setminus \bar{\Omega}_1}|V|^2 |f|^2 d\x.
\]
Guided by the proof of Corollary 1 in \cite{AL}, the theorem
follows on setting
\[
V= 2(n-1)\frac{\nabla |\x|}{|\x|}-2\frac{\nabla
\delta_1}{\delta_1} -2\frac{\nabla \delta_2}{\delta_2}
\]
and $ \varepsilon = 2.$
\end{proof}

If $\Omega_1 = B_{\rho}, \Omega_2 = B_R,\ R> \rho,$ we have
\[
\Delta \delta_1 = \frac{n-1}{|\x|},\ \ \ \Delta \delta_2 =
-\frac{n-1}{|\x|}
\]
by Lemma \ref{Lem1}, and $\nabla \delta_1 = - \nabla \delta_2 =
\x/|\x|.$ On substituting in (\ref{Eq6}), we derive Corollary 1 in
\cite{AL}, namely,
\begin{equation}\label{Eq7}
    \int_{B_R \setminus B_{\rho}} |\nabla f(\x)|^2 d \x
   \ge
    \frac{1}{4}\int_{B_R \setminus B_{\rho}} \left
    \{\frac{(n-1)(n-3)}{|\x|^2} +\frac{1}{\delta_1^2} +
    \frac{1}{\delta_2^2}+ \frac{2}{\delta_1 \delta_2} \right \}
    |f(\x)|^2 d\x,
\end{equation}
where $\delta_1(\x) = |\x|-\rho, \delta_2(\x) = R-|\x|.$

\bigskip

\section{Non-convex domains}
\subsection{Torus}
We show that Theorem \ref{ThmP2} can be applied to give a
Hardy-type inequality on a torus.

\begin{Corollary}\label{Torus} Let $\Omega\subset\R^3$ be the interior of a ring torus with
minor radius $r$ and major radius $R > 2r$. Then $\Delta\delta <
0$ in $\Omega\setminus \mathcal{R}(\Omega)$ and
\begin{eqnarray}\label{EE2}
\int_\Omega |\grad \delta\cdot \grad f|^pd\x &\ge&
\left(\frac{p-1}{p}\right)^p \int_\Omega \frac{|f|^p}{\delta^p}d\x
\nonumber \\
& +& \left(\frac{p-1}{p}\right)^{p-1}\int_\Omega \left(
\frac{1}{(r-\delta)} - \frac{1}{\sqrt{x_1^2 + x_2^2}}
\right)\frac{|f|^p}{\delta^{p-1}}d\x \nonumber \\
\end{eqnarray}
for all $f\in\Con(\Omega)$, where $\x \in \Omega$ has co-ordinates
$(x_1, x_2,x_3)$, and the last integrand is positive.
\end{Corollary}
\begin{proof} The domain $\Omega$ under consideration is the ``doughnut-shaped" domain generated by rotating a disc of radius $r$ about a co-planar axis at a distance $R$ from the center of the disc.
The fact that $\Delta\delta\le 0$ was proved by D.H.~Armitage and
\"U.~Kuran~\cite{AK}. We give a different proof here which meets
our purposes using a curvature argument.

The ridge of the torus is
$$
 \mathcal{R}(\Omega)=\{\x: \rho(\x)=0\},
$$
where $\rho(\x)$ is the distance from the point $\x$ in $\Omega$
to the center of the cross-section and $\delta(\x)=r-\rho(\x)$.
Moreover, in the notation of Theorem \ref{ThmP2},
$$
S_\epsilon =\{\x:\rho(\x)<\epsilon\},
$$
and points on the surface of $ S_\epsilon$ are on the level
surface $\rho(\x)=\epsilon$, so that the unit inward normal to
$\pd S_{\varepsilon}$ is
$\eta_{\varepsilon}=-\grad\rho(\x)/|\grad\rho(\x)|=\grad\delta$.
Therefore $\grad \delta\cdot\grad \eta >0$.

For $ \x \in \Omega \setminus \mathcal{R}(\Omega),$ let $\y =
N(\x)= (y_1,y_2,y_3)$ have the parametric co-ordinates
$$
\begin{array}{l}
y_1=(R+r\cos s^2)\cos s^1\\
y_2=(R+r\cos s^2)\sin s^1\\
y_3=r\sin s^2\end{array},
$$
where $s^1,s^2\in (-\pi,\pi]$. The principal curvatures at $\y \in
\pd \Omega$ are
$$
\kappa_1 =-\frac{1}{r}, \q \kappa_2=-\frac{\cos s^2}{R+r\cos s^2},
$$
e.g., see Kreyszig~\cite{K}, p.135, and so, by Lemma 1,
\begin{eqnarray*}
\Delta\delta(\x)=\sum_{i=1}^{2}\left(\frac{\kappa_i}{1+\delta\kappa_i}\right)(\y)
&=& -\frac{R+2(r-\delta)\cos s^2}{(r-\delta)(R+(r-\delta)\cos s^2)
}\\
&=& \ -\frac{\sqrt{x_1^2 +x_2^2} + (r-\delta)\cos
s^2}{(r-\delta)\sqrt{x_1^2 + x_2^2}} < 0
\end{eqnarray*}
since $ R + r \cos s^2 = \sqrt{x_1^2 + x_2^2} + \delta(\x) \cos
s^2$ and $R> 2r.$ The inequality (\ref{EE2}) follows from Theorem
\ref{ThmP2}.

\end{proof}

\bigskip

\subsection{1-sheeted Hyperboloid}
Next, we apply Theorem \ref{ThmP} to the 1-sheeted hyperboloid
\begin{equation} \label{sh}
\Omega=\{(x_1,x_2,x_3)\in\R^3: x_1^2+x_2^2<1+x_3^2\}.
\end{equation}
This is non-convex and unbounded with infinite volume and infinite
interior diameter $D_{int}(\Omega)$. To calculate the principal
curvatures, we choose the following parametric co-ordinates for
$\y \in \partial\Omega$:
$$\begin{array}{rl}
y_1(s,t)=&\sqrt{s^2+1}\cos t,\\
y_2(s,t)=&\sqrt{s^2+1}\sin t,\\
y_3(s,t)=& s,
\end{array}
$$
for $t\in [0,2\pi)$ and $s\in (-\infty,\infty)$. A calculation
then gives (see \cite{K}, p. 132)
$$
\kappa_1= - \frac{1}{[2s^2+1]^{3/2}}, \q
 \kappa_2=
\frac{1}{\sqrt{2s^2+1}},
$$
and if $\y = N(\x),\ \x \in \Omega\setminus \mathcal{R}(\Omega) $,
then by Lemma \ref{Lem1},
\begin{equation} \label {tk}\Delta\delta
(\x)=\tilde\kappa:=\sum_{i=1}^2\frac{\kappa_i}{1+\delta\kappa_i}=-\frac{1}{w^3
- \delta} + \frac{1}{w+\delta},
\end{equation}
where $w = \sqrt{2s^2 +1}$ is the distance of $\y$ from the
origin, and the ridge is
$\mathcal{R}(\Omega)=\{(x_1,x_2,x_3):x_1=x_2=0,\ x_3\in
(-\infty,\infty)\}$. Therefore $\Delta\delta (\x)$ changes sign in
$\Omega.$

To find $\y = N(\x)$, we first determine the vector normal to $\pd
\Omega$ at $\y$, namely
\begin{eqnarray*}
\y_s \times \y_t &=& \left|\begin{array}{ccc} i & j & k \\
\frac{s}{\sqrt{s^2+1}} \cos t & \frac{s}{\sqrt{s^2+1}} \sin t & 1
\\ - \sqrt{s^2+1} \sin t &  \sqrt{s^2+1} \cos t & 0 \end{array}\right
| \\
&=& [-\sqrt{s^2+1} \cos t] i + [-\sqrt{s^2+1} \sin t] j + s k.
\end{eqnarray*}
The inward unit normal vector at $\y$ is therefore
\[
{\bf{n}} = \{[-\sqrt{s^2+1} \cos t] i + [-\sqrt{s^2+1} \sin t] j +
s k \}/ \sqrt{2s^2+1}.
\]
The distance from $\y$ to the ridge point $p(\x)$ of $\x$ (see
Section 2) is given by $ \sqrt{s^2+1}/ \cos \theta,$ where $ \cos
\theta = ({\bf{z}} \cdot {\bf{n}}) /|{\bf{z}}|,$ and
\[
\z = [-\sqrt{s^2+1} \cos t] i + [-\sqrt{s^2+1} \sin t] j.
\]
Hence
\[ \sqrt{s^2+1}/ \cos \theta = \sqrt{2s^2 +1} = w.
\]
Consequently, the near point of $\x$ is the point on the boundary
of $\Omega$ which is equidistant from the ridge point $p(\x)$ of
$\x$ and the origin.

We therefore have from Theorem \ref{ThmP}.

\begin{Corollary} Let $\Omega\subset\R^3$ be the 1-sheeted
hyperboloid (\ref{sh}). Then, for all $f\in\Con(\Omega\setminus
\mathcal{R}(\Omega)),$
\begin{equation}\label{hyperboloid}\begin{array}{rl}
\int_\Omega |\grad \delta\cdot \grad f|^pd\x \ge&
\left(\frac{p-1}{p}\right)^p \int_\Omega \frac{|f|^p}{\delta^p}d\x
- \left(\frac{p-1}{p}\right)^{p-1}\int_\Omega \tilde\kappa
\frac{|f|^p}{\delta^{p-1}}d\x,
\end{array}
\end{equation}
where $\tilde{\kappa}$ is given in (\ref{tk}), with $w = |\y| =
\delta (p(\x))$, $\y = N(\x) $ and $p(\x)$ the ridge point of
$\x$.
\end{Corollary}

\bigskip

\subsection{Doubly connected domains}
A domain $\Omega \subset \R^2 \equiv \C$ is \textit{doubly
connected} if its boundary is a disjoint union of 2 simple curves.
If it has a smooth boundary then it can be mapped conformally onto
an annulus $ \Omega_{\rho,R} = B_R \setminus \overline{B_{\rho}} =
\{z \in \C: \rho<|z| < R\},$ for some $\rho,R;$ see \cite{Wen},
Theorem 1.2.
\begin{Lemma}\label{Lem2} Let $\Omega_1\subset\Omega_2\subset
\C$ and $B_\rho\subset B_R\subset\C$, $0<\rho<R$, where $B_r$ is
the disc of radius $r$ centered at the origin.  Let
$$
F:\Omega_2\setminus\bar\Omega_1\to B_R\setminus \overline{B_\rho}
$$
be analytic and univalent. Then for $\z=x_1+ix_2$,
$\x=(x_1,x_2)\in\Omega_2\setminus \bar{\Omega}_1$,
\begin{equation} \label{E2}
\mathfrak{F}(\z):= -\frac{|F'(\z)|^2}{|F(\z)|^2}+|F'(\z)|^2\left
\{\frac{1}{|F(\z)|-\rho}+\frac{1}{R-|F(\z)|}\right\}^2
\end{equation}
is invariant under scaling, rotation, and inversion. Hence,
$\mathfrak{F}$ does not depend on the choice of the mapping $F$,
but only on the geometry of $\Omega_2\setminus\bar{\Omega}_1$.
\end{Lemma}

\begin{proof}
The fact that $\mathfrak F$ is invariant under scaling and
rotations is straightforward. To see that it is also invariant
under inversions suppose that $F(\z)=1/G(\z)$. Then, under
inversion $\mathfrak F(\z)$ becomes
$$\begin{array}{l}
-\frac{|G'(\z)|^2}{|G(\z)|^2}+\frac{|G'(\z)|^2}{|G(\z)|^4}
\left\{\frac{1}{\frac{1}{|G(\z)|}-\rho^{-1}}+\frac{1}{R^{-1}-\frac{1}{|G(\z)|}}\right\}^2\\
=-\frac{|G'(\z)|^2}{|G(\z)|^2}+\frac{|G'(\z)|^2}{|G(\z)|^2}
\left\{\frac{\rho}{\rho-|G(\z)|}+\frac{R}{|G(\z)|-R}\right\}^2\\
=-\frac{|G'(\z)|^2}{|G(\z)|^2}+\frac{|G'(\z)|^2}{|G(\z)|^2}
\left\{\frac{(\rho-R)|G(\z)|}{(\rho-|G(\z)|)(|G(\z)|-R)}\right\}^2\\
=-\frac{|G'(\z)|^2}{|G(\z)|^2}+|G'(\z)|^2
\left\{\frac{1}{\rho-|G(\z)|}+\frac{1}{|G(\z)|-R}\right\}^2
\end{array}
$$
implying that $\mathfrak{F}$ is invariant under inversions. The
rest of the lemma follows from \cite{Krantz}, p. 133.

\end{proof}

In applying the last Lemma we regard $\Omega_1$, $\Omega_2$ as
domains in $\R^2$ with $\z=x+iy$ and $\x=(x,y)$.
\begin{Theorem}\label{Thm3}
For $\Omega:=\Omega_2\setminus\bar{\Omega}_1\subset\R^2$,
$$
\int_\Omega |\grad u(\x)|^2d\x \ge \frac14 \int_\Omega
\mathfrak{F}(\x)|u(\x)|^2d\x.
$$
\end{Theorem}
\begin{proof}
From Corollary~1 of \cite{AL}  it follows that for all $u\in
H_0^1(B_R\setminus \overline{B_\rho})$,
$$
\int_{B_R\setminus \overline{B_\rho}}|\grad u(\y)|^2d\y \ge \frac14
\int_{B_R\setminus \overline{B_\rho}}\left
[\frac{-1}{|\y|^2}+\left(\frac{1}{\delta_\rho(\y)}
+\frac{1}{\delta_R(\y)}\right)^2\right]|u(\y)|^2d\y,
$$
where $\delta_\rho(\y):=|\y|-\rho$ and $\delta_R(\y):=R-|\y|$. Let $F: \Omega \rightarrow \Omega_{\rho,R}$ be analytic and univalent, and set
$\y=F(\x)$, with $\y=(y_1,y_2)$, $\x=(x_1,x_2)$. Then, with $F'$
denoting the complex derivative,
$$
d\y=\left|
\det\left(\frac{\partial(y_1,y_2)}{\partial(x_1,x_2)}\right
)\right|d\x=|F'(\x)|^2d\x,
$$
$$
\grad_\x u=\grad_\y
u\left[\frac{\partial(y_1,y_2)}{\partial(x_1,x_2)}\right]^t,
$$
implying that
$$
|\grad_\x u|^2=|\grad_\y u|^2|F'(\x)|^2.
$$
The theorem follows from Lemma \ref{Lem2}.

\end{proof}

\bigskip

\begin{Example} Let $\Phi(z) = (z-1)(z+1)$ and
\[
\Omega = \{z: \rho^2 < |\Phi(z)| < R^2 \}
\]
for $0<\rho <R.$ The function $F(z)= \sqrt{\Phi(z)}$ is analytic
and univalent in $\Omega$ and
\[
F: \Omega \rightarrow \Omega_{\rho,R}.
\]
A calculation gives
\begin{eqnarray*}
  \mathfrak{F}(z) &=& -\frac{|z|^2}{|z^2-1|^2} \\
   &+& \frac{|z|^2}{|z^2-1|}\frac{(R-\rho)^2}{(\sqrt{|z|^2-1} -
   \rho)^2(R-\sqrt{|z|^2-1})^2}.
\end{eqnarray*}
\end{Example}

Finally, we refer the reader to further developments along these lines in \cite{LLL}.

\bibliographystyle{amsalpha}

\end{document}